\documentclass[11pt,reqno]{amsart}
\usepackage[T1]{fontenc}
\usepackage{lmodern}
\usepackage{amsmath,amssymb,amsthm,mathtools}
\usepackage[margin=1.15in]{geometry}
\usepackage[protrusion=true,expansion=false]{microtype}
\usepackage[hidelinks]{hyperref}
\makeatletter
\def\ps@headings{\ps@empty
  \def\@evenhead{%
    \setTrue{runhead}%
    \normalfont\fontsize{8}{9.5}\selectfont
    \thepage\enspace\hfill
    \def\thanks{\protect\thanks@warning}%
    \leftmark{}{}}%
  \def\@oddhead{%
    \setTrue{runhead}%
    \normalfont\fontsize{8}{9.5}\selectfont
    \def\thanks{\protect\thanks@warning}%
    \rightmark{}{}\hfill\enspace\thepage}%
  \let\@mkboth\markboth
}
\makeatother
\numberwithin{equation}{section}
\newtheorem{theorem}{Theorem}[section]
\newtheorem{proposition}[theorem]{Proposition}
\newtheorem{lemma}[theorem]{Lemma}

\newcommand{\C}{\mathbb C}

\title[Griffiths positivity does not imply positivity of the top Chern form]
 {Griffiths positivity does not imply positivity of the top Chern form}
\author{Yun-Heng Du}
\address{Academy of Mathematics and Systems Science, Chinese Academy of Sciences, Beijing 100190, China}
\email{duyunheng22@mails.ucas.ac.cn}
\date{}
\hypersetup{pdftitle={Griffiths positivity does not imply positivity of the top Chern form},
 pdfauthor={Yun-Heng Du}}
\begin{document}
\raggedbottom
\begin{abstract}
For every integer $r\geq9$, we construct a smooth Griffiths-positive
Hermitian metric on $\mathcal O_{\mathbb P^r}(1)^{\oplus r}$ whose top
Chern form is pointwise negative on a nonempty open set. This gives
counterexamples on compact projective manifolds to Griffiths' conjecture
on the positivity of Chern--Weil forms.
\end{abstract}
\maketitle
\markboth{\shorttitle}{\shorttitle}
\section{Introduction}

\subsection{Chern--Weil positivity and the main theorem}

Let $X$ be a complex manifold of dimension $n$, and let $(E,h)$ be a
Hermitian holomorphic vector bundle of rank $r$ on $X$. The metric $h$ is
Griffiths positive if the curvature $R_h$ of its Chern connection
satisfies
\[
 \langle R_h(u,\bar u)v,v\rangle_h>0
\]
for every $x\in X$ and all nonzero $u\in T_x^{1,0}X$ and $v\in E_x$.
We normalize the Chern forms by
\[
 \det\left(I_r+s\frac{i}{2\pi}R_h\right)
   =\sum_{q=0}^r s^q c_q(E,h).
\]
Assign weight $q$ to $c_q$. Fulton and Lazarsfeld
\cite{FultonLazarsfeld1983} identified the cone of homogeneous positive
polynomials considered by Griffiths with the nonzero polynomials
\[
 P=\sum_{\substack{\lambda\vdash k\\\ell(\lambda)\leq r}}
 a_\lambda s_\lambda,
 \qquad a_\lambda\geq0,
\]
where $s_\lambda$ is the Schur polynomial in the $r$ formal Chern roots,
$\ell(\lambda)$ is the number of parts of $\lambda$, and the coefficients
are not all zero. We use the convention in which $c_j$ is the $j$th
elementary symmetric polynomial in these roots; thus $s_{(1^r)}=c_r$.

In 1969, Griffiths \cite[p.~247]{Griffiths1969} proposed what is now
called Griffiths' conjecture on the positivity of Chern--Weil forms. It
asserts that, if $X$ is a compact complex manifold of dimension $n$ and
$(E,h)$ is a Griffiths-positive Hermitian holomorphic vector bundle of
rank $r$ on $X$, then for every integer $k$ with $1\leq k\leq n$ and
every nonzero positive polynomial $P\in\mathbb Q[c_1,\ldots,c_r]$ of
weighted degree $k$, the Chern--Weil form
\[
 P\bigl(c_1(E,h),\ldots,c_r(E,h)\bigr)
\]
is Hermitian positive at every point of $X$. Here a $(k,k)$-form is
Hermitian positive if at each point it can be written as
$\sum_j i^{k^2}\xi_j\wedge\overline{\xi_j}$ for $(k,0)$-covectors $\xi_j$.
Equivalently, the positive cone spanned by the Schur forms of $(E,h)$
consists of Hermitian positive differential forms, with Hermitian
positivity understood as in Fagioli
\cite[Remark~1.4 and Proposition~1.7]{Fagioli2022}.

The theorem below gives a counterexample with $X=\mathbb P^r$,
$k=n=r$, and $P=c_r$. Here $c_r$ is a Schur polynomial, and all standard
positivity cones coincide in top degree; the pointwise negativity in the
theorem therefore contradicts the conjectured Hermitian positivity.

\begin{theorem}\label{cor:ext-geometric}
For every integer $r\geq9$, the holomorphic vector bundle
$E_r\coloneqq\mathcal O_{\mathbb P^r}(1)^{\oplus r}$ admits a smooth Hermitian
metric $h_r$ which is Griffiths positive on $\mathbb P^r$ and whose top
Chern form $c_r(E_r,h_r)$ is pointwise negative on a nonempty open subset
of $\mathbb P^r$.
\end{theorem}

Thus Griffiths' conjecture on the positivity of Chern--Weil forms fails
on compact projective manifolds, even for a direct sum of an ample line
bundle. The smallest rank for which the top Chern form of a
Griffiths-positive metric can fail to be weakly nonnegative remains open.

Griffiths \cite{Griffiths1969} proved positivity
of the second Chern form in rank two. Fagioli \cite{Fagioli2022}
established weak nonnegativity of the second Chern form in arbitrary rank
under Griffiths semipositivity and of $c_1c_2-c_3$ in rank three. Wan
\cite{Wan2026} proved weak positivity of the third Chern form in arbitrary
rank and positivity of all Schur forms for rank-three Griffiths-positive
bundles over threefolds. Guler \cite{Guler2012} proved positivity of the
signed Segre forms for Griffiths-positive bundles on projective
manifolds. Together, these results give affirmative cases in each of the
first three Chern degrees;
the counterexample below arises from a rank-nine construction.

Stronger curvature hypotheses also ensure positivity of characteristic
forms. The curvature condition now called Bott--Chern nonnegativity gives
nonnegative Chern forms by Bott and Chern \cite{BottChern1965}; Li
\cite{Li2021} extended the conclusion to Schur forms for Bott--Chern
nonnegative metrics on compact complex manifolds. Finski
\cite{Finski2022} proved positivity of all Schur forms under Nakano or
dual Nakano positivity. Wan \cite{Wan2024} obtained weak positivity under
strongly decomposable positivity of type I and positivity under strongly
decomposable positivity of type II. Thus the failure in
Theorem~\ref{cor:ext-geometric} already occurs under Griffiths positivity.

The numerical theory concerns Chern numbers and cohomology classes on
projective varieties. Bloch and Gieseker \cite{BlochGieseker1971}
established numerical positivity of the Chern classes of ample vector
bundles on projective varieties. Fulton and Lazarsfeld
\cite{FultonLazarsfeld1983} proved that the positive polynomials described
above are precisely the homogeneous numerically positive polynomials
for ample vector bundles. The theorem below concerns the sign of a
Chern--Weil representative at each point, for a specified metric.

\subsection{Reformulation and construction strategy}

The pointwise curvature problem has a matrix formulation in terms of
positive linear maps. For $n\geq1$, write $M_n=M_n(\mathbb C)$.
Transpose and conjugate transpose are denoted by $\mathsf T$ and $*$,
respectively. A complex-linear map $\Phi:M_n\to M_q$ is positive
(respectively, strictly positive) if $\Phi(H)\geq0$ for every $H\geq0$
(respectively, $\Phi(H)>0$ for every nonzero $H\geq0$). By spectral
decomposition, positivity is equivalent to the rank-one test
\begin{equation}\label{eq:intro-rank-one-test}
 v^*\Phi(uu^*)v\geq0
 \qquad (u\in\mathbb C^n,\ v\in\mathbb C^q).
\end{equation}
Finski \cite[Theorem~1.3]{Finski2022} reduced the universal problem of
weak positivity of Schur forms to a finite-dimensional problem for positive
linear maps. When the source and target matrix spaces are both $M_n$ and
one considers the top Chern form in complex dimension $n$, the scalar
takes the following form. For a matrix variable
$Q=(q_{ij})\in M_{r,s}(\mathbb C)$, whose entries are independent complex
variables, write
\[
 \partial_Q\coloneqq
 \left(\frac{\partial}{\partial q_{ij}}\right)_{
  1\leq i\leq r,\,1\leq j\leq s}.
\]
For a square matrix variable, $\det(\partial_Q)$ is the determinant of
this matrix of commuting differential operators.
With $Y=(y_{ij})$, the scalar associated with a complex-linear
endomorphism $\Phi$ of $M_n$ is
\begin{equation}\label{eq:F-definition}
 F_n(\Phi)\coloneqq\det(\partial_Y)\det\Phi(Y).
\end{equation}
For a Hermiticity-preserving map $\Phi:M_n\to M_n$, consider near the
origin the Hermitian metric
\[
 h_\Phi(z)\coloneqq I_n-2\pi\Phi(zz^*).
\]
At the origin,
$\langle R^{h_\Phi}(u,\bar u)v,v\rangle_{h_\Phi}|_0
=2\pi v^*\Phi(uu^*)v$ for $u,v\in\mathbb C^n$, and its top Chern form is
\begin{equation}\label{eq:intro-top-chern-test}
 c_n(\mathbb C^n\times\mathbb C^n,h_\Phi)\big|_0
 =F_n(\Phi)\,dV_n\big|_0.
\end{equation}
Here $dV_n=\bigwedge_{j=1}^n(i\,dz_j\wedge d\bar z_j)$.
Thus the geometric problem leads to the construction of a strictly
positive map for which $F_n$ is negative.

The same formulation contains the known two- and three-dimensional cases. Finski
\cite[Proposition~1.4(c)]{Finski2022} proved the required nonnegativity
when $n=2$. Wan \cite[Theorem~0.3]{Wan2026} proved it when $n=3$.
His argument first uses operator scaling to pass, without changing the
sign of $F_3$, to maps that are unital and trace preserving
\cite[pp.~8--9]{Wan2026}. It then expresses this
scalar through spherical moment identities involving traces of products
of matrices and applies a symmetric spectral inequality
\cite[Lemmas~2.6 and~2.9 and Theorem~2.11]{Wan2026}. In dimension four,
the analogous expansion contains an additional sum of traces indexed by
the $4$-cycles, which complicates the extension of this spectral argument
\cite[Remark~2.12]{Wan2026}.

After constructing a rank-two ample bundle with no smooth
Griffiths-semipositive metric \cite{DuXie2026}, Du and Xie developed a
measure-theoretic criterion for Griffiths-positive metrics on smooth
projective varieties in
\cite{DuXieFirstJet2026}. Condition~(iv) characterizes the existence of a
Griffiths-positive metric by the absence of a nonzero positive
matrix-valued measure annihilating the relevant Levi forms. Its
finite-dimensional version produces, when separation fails, an atomic
probability measure under which matrices formed from first jets have zero
average \cite[Proposition~2.1]{DuXieFirstJet2026}. This formulation led
the author to study the present positivity problem through expectations
of matrix-valued data. Hoeffding's distribution comparisons for sums of independent Bernoulli
variables \cite{Hoeffding1956}, Darroch's bound relating their mean and
mode \cite{Darroch1964}, and Gnedin's extension of the mean--mode relation
\cite[Theorem~2]{Gnedin2024} provided initial models for extracting sign
information from expectations. The formula for $F_n(\Phi)$ called for a
matrix distribution under which the two determinant expansions, with
their row and column permutations and signs, can be evaluated together.

Complex Gaussian matrices supply such a law. If $G_n=(g_{ij})$ is an
$n\times n$ matrix of independent standard complex Gaussian entries,
normalized by $\mathbb E|g_{ij}|^2=1$, then
\begin{equation}\label{eq:gaussian-F}
 F_n(\Phi)
 =\mathbb E\bigl[\overline{\det G_n}\det\Phi(G_n)\bigr].
\end{equation}
This identity is a specialization of Bargmann's classical Gaussian
integral representation of the apolar inner product \cite{Bargmann1961};
Aldaz, Bravo, and Render \cite[Theorem~7]{AldazBravoRender2025} give a
modern polynomial formulation.
Independence and the orthogonality of complex Gaussian monomials pair
the coefficients of the two determinants in \eqref{eq:gaussian-F}.
Gaussian integration by parts then transfers conjugate factors to
derivatives, allowing the terms selected by the construction to be
expressed as weighted moments. Lemma~\ref{prob:gaussian-apolarity}
establishes these identities; the integration-by-parts formula is the
complex polynomial counterpart of Stein's Gaussian identity
\cite[Lemma~1]{Stein1981}.

Wan's rank-three argument provides a closely related use of moment identities
\cite[Lemmas~2.6 and~2.9]{Wan2026}. In rank three, the trace terms
produced by the cycle expansion combine into elementary symmetric
polynomials of the positive matrix associated with each sphere vector,
so a spectral inequality proves nonnegativity. Our use of the Gaussian
pairing in \eqref{eq:gaussian-F} retains the simultaneous row and column
correlations in the two determinant expansions. Orthogonality and
integration by parts express their contribution through weighted moments,
which the Cayley identities evaluate; these moments are not controlled by
the rank-one positivity tests, and for a suitable map the rank-nine top
Chern coefficient is negative.

Equation~\eqref{eq:gaussian-F} characterizes the difference between the
two positivity properties at the level of one curvature tensor. Griffiths
semipositivity is nonnegativity of the rank-one quadratic tests in
\eqref{eq:intro-rank-one-test} on every decomposable tangent--fiber
direction. Positivity of the top Chern form is positivity of the
scalar $F_n(\Phi)$ in \eqref{eq:intro-top-chern-test}, and
\eqref{eq:gaussian-F} writes this scalar as
$\mathbb E[\overline{\det G_n}\det\Phi(G_n)]$. Thus positivity of the map
controls $\Phi$ on rank-one positive semidefinite matrices, whereas the
sign of the top Chern form depends on the signed sum obtained by expanding
both determinants. The expectation identifies the monomials in
$\det\Phi(Y)$ that contribute to $F_n(\Phi)$. The construction makes their
total contribution negative while preserving the rank-one inequalities in
\eqref{eq:intro-rank-one-test}. Orthogonality and integration by parts
express this contribution through weighted moments, which the Cayley
identities evaluate.

\subsection{Construction of the counterexample}

The proof of Theorem~\ref{cor:ext-geometric} starts with a positive map
for which $F_9<0$. The Gaussian representation identifies the terms
responsible for the negative sign.
A small trace perturbation makes the map strictly positive while
preserving this sign. We then realize it as the curvature of a local
metric and extend the metric to projective space, preserving its
curvature near the origin and Griffiths positivity everywhere.

The positive three-dimensional reduction map
\[
 R:M_3\to M_3,\qquad R(A)\coloneqq\operatorname{tr}(A)I_3-A,
\]
will be used below, and $F_3(R)=0$.
The six-dimensional Breuer--Hall map of Breuer \cite{Breuer2006} and Hall
\cite[Theorem~4]{Hall2006} uses two copies of $R$ and skew-symmetric
off-diagonal blocks while preserving positivity. We append a
three-dimensional diagonal block whose determinant weights the
contribution of the two skew-symmetric blocks. Since $F_3(R)=0$, the terms
of $F_9$ independent of these blocks vanish, so their joint contribution
alone determines the sign of $F_9$; this is what leads from the
six-dimensional Breuer--Hall map to rank nine.

For the resulting family $\Phi_t:M_9\to M_9$, $t\geq0$, Gaussian
integration by parts expresses that contribution through
\[
 \mathbb E[|\det G|^2|\operatorname{tr}G|^2]
 -\mathbb E[|\det G|^2\|G\|_{\mathrm F}^2],
\]
where $G$ is a $3\times3$ matrix with independent standard complex
Gaussian entries and $\|G\|_{\mathrm F}^2=\sum_{i,j}|g_{ij}|^2$.
The trace involves the diagonal entries, whereas the Frobenius norm
includes every entry. The Cayley identities evaluate these weighted
moments and give
\[
 F_9(\Phi_t)=-1536t^3
 \qquad(t\geq0).
\]
Thus, for $t>0$, the map satisfies $v^*\Phi_t(uu^*)v\geq0$ for all
$u,v$ while the expectation in \eqref{eq:gaussian-F} is negative. By
\eqref{eq:gaussian-F}, its exact value is $-1536t^3$.

The same block determinant calculation extends the construction to every
rank $r\geq9$. Proposition~\ref{geo:strict-map} shows that adjoining scalar
diagonal blocks and adding a small trace term produces a strictly positive
map $\Lambda_r:M_r\to M_r$ with $F_r(\Lambda_r)<0$.

The metric $h_{\Lambda_r}$ on a neighborhood of the origin realizes the
prescribed curvature there, as in Finski \cite[Proposition~2.10]{Finski2022}. After
pairing with the metric, the curvature is the sum of positive rank-one
evaluations of the map and a positive semidefinite correction, so the
metric is Griffiths positive near the origin. The base-point
formula \eqref{eq:intro-top-chern-test} gives its top Chern form there as
$F_r(\Lambda_r)dV_r$, which is negative.

At the origin, the local metric has the same value and first derivatives
as the Fubini--Study direct sum metric on
$\mathcal O_{\mathbb P^r}(1)^{\oplus r}$, so their difference vanishes to
second order. A cutoff varying slowly on a logarithmic radial scale joins
the two metrics in one step; the resulting curvature errors are dominated
by the uniform Griffiths lower bound along the affine segment between them.
Proposition~\ref{geo:projective-realization} gives this extension for the
metric $I_q-2\pi\Lambda(zz^*)$ associated with any strictly positive
linear map $\Lambda:M_n\to M_q$.

Section~\ref{sec:construction} develops a positive-map formulation of the
curvature at a point, represents the top Chern coefficient by a Gaussian
moment, and constructs in rank nine a local Hermitian metric whose curvature
is Griffiths semipositive and whose top Chern form is negative at the origin.
Section~\ref{sec:geometric-counterexample} makes the underlying map strictly
positive in every rank $r\geq9$ and realizes its curvature by a globally
Griffiths-positive metric on projective space, producing the compact
counterexamples.

\section{Positive maps and top Chern forms}
\label{sec:construction}\label{alg:section}

In this section, we establish the pointwise part of the construction: a
Hermitian metric whose curvature is Griffiths semipositive at the origin
and whose top Chern form is negative there. This construction will be used
to obtain compact counterexamples.

For $n\geq1$, write $M_n=M_n(\mathbb C)$ for the space of $n\times n$
complex matrices, and let $\Phi:M_n\to M_n$ be a complex-linear map
satisfying $\Phi(A^*)=\Phi(A)^*$ for every $A\in M_n$. On $\mathbb C^n$, use coordinates
$z_j=x_j+iy_j$. For a Hermitian metric matrix $h$, we use the convention
\begin{equation}\label{geo:curvature-convention}
 \nabla=d+h^{-1}\partial h,\qquad
 R^h=\bar\partial(h^{-1}\partial h),\qquad
 \Omega^h=\frac{i}{2\pi}R^h.
\end{equation}
Writing $h_i\coloneqq\partial h/\partial z_i$ and similarly for the other
derivatives gives
\begin{equation}\label{geo:curvature-coefficients}
 R^h_{i\bar j}
 =-\partial_{\bar j}(h^{-1}h_i)
 =h^{-1}h_{\bar j}h^{-1}h_i-h^{-1}h_{i\bar j}.
\end{equation}

The Hermitian matrix
\begin{equation}\label{geo:quadratic-metric}
 h_\Phi(z)=I_n-2\pi\Phi(zz^*)
\end{equation}
satisfies $h_\Phi(0)=I_n$ and is therefore positive definite near the
origin. Its first derivatives satisfy $dh_\Phi(0)=0$. With $E_{ij}$ denoting
the standard matrix units in $M_n$,
its mixed second derivatives are
$(h_\Phi)_{i\bar j}=-2\pi\Phi(E_{ij})$.
Equation~\eqref{geo:curvature-coefficients} therefore gives
\begin{equation}\label{geo:quadratic-curvature}
 \Omega^{h_\Phi}|_0
 =-\frac{i}{2\pi}\sum_{i,j=1}^n
   (h_\Phi)_{i\bar j}(0)\,dz_i\wedge d\bar z_j\big|_0
 =i\sum_{i,j=1}^n\Phi(E_{ij})\,dz_i\wedge d\bar z_j\big|_0.
\end{equation}
Use the positive volume form
\begin{equation}\label{geo:volume}
 dV_n=\bigwedge_{j=1}^n(i\,dz_j\wedge d\bar z_j)
     =2^n\,dx_1\wedge dy_1\wedge\cdots\wedge dx_n\wedge dy_n.
\end{equation}
Taking the determinant in \eqref{geo:quadratic-curvature} gives the complete
expansion of the top Chern form at the origin:
\[
 \begin{aligned}
 c_n(\mathbb C^n\times\mathbb C^n,h_\Phi)|_0
 &=\sum_{\sigma\in S_n}\operatorname{sgn}(\sigma)
   \bigwedge_{a=1}^n\left(
     i\sum_{i_a,j_a=1}^n
     \Phi(E_{i_a j_a})_{a,\sigma(a)}
     dz_{i_a}\wedge d\bar z_{j_a}\right) \\
 &=\sum_{\sigma,\rho,\eta\in S_n}\operatorname{sgn}(\sigma)
   \prod_{a=1}^n\Phi(E_{\rho(a),\eta(a)})_{a,\sigma(a)}
   \bigwedge_{a=1}^n
     (i\,dz_{\rho(a)}\wedge d\bar z_{\eta(a)}) \\
 &=\sum_{\sigma,\rho,\eta\in S_n}
   \operatorname{sgn}(\sigma)\operatorname{sgn}(\rho)
   \operatorname{sgn}(\eta)
   \prod_{a=1}^n
   \Phi(E_{\rho(a),\eta(a)})_{a,\sigma(a)}\,dV_n|_0.
 \end{aligned}
\]
Now let $Y=(y_{ij})$ be an $n\times n$ matrix of independent variables and
write
\[
 \partial_Y\coloneqq
 \left(\frac{\partial}{\partial y_{ij}}\right)_{i,j=1}^n.
\]
Then
\begin{align*}
 F_n(\Phi)
 &=\det(\partial_Y)\det\Phi(Y)\\
 &=\sum_{\sigma,\tau\in S_n}
   \operatorname{sgn}(\sigma)\operatorname{sgn}(\tau)
   \left(\prod_{i=1}^n\frac{\partial}{\partial y_{i,\tau(i)}}\right)
   \prod_{a=1}^n\left(\sum_{p,q=1}^n
   y_{pq}\Phi(E_{pq})_{a,\sigma(a)}\right)\\
 &=\sum_{\sigma,\tau,\rho\in S_n}
   \operatorname{sgn}(\sigma)\operatorname{sgn}(\tau)
   \prod_{a=1}^n
   \Phi(E_{\rho(a),\tau(\rho(a))})_{a,\sigma(a)}\\
 &=\sum_{\sigma,\rho,\eta\in S_n}
   \operatorname{sgn}(\sigma)\operatorname{sgn}(\rho)
   \operatorname{sgn}(\eta)
   \prod_{a=1}^n
   \Phi(E_{\rho(a),\eta(a)})_{a,\sigma(a)}.
\end{align*}
Comparing the two expansions yields
\begin{equation}\label{geo:quadratic-chern}
 c_n(\mathbb C^n\times\mathbb C^n,h_\Phi)|_0
 =F_n(\Phi)\,dV_n|_0.
\end{equation}
We call $\Phi$ positive if $\Phi(H)\geq0$ for every $H\geq0$.
Here $A\geq0$ means that $A$ is a positive semidefinite Hermitian matrix.
Equation~\eqref{geo:quadratic-curvature} shows that Griffiths
semipositivity of $h_\Phi$ at the origin is equivalent to
$\Phi(uu^*)\geq0$ for every $u\in\mathbb C^n$. Since every $H\geq0$ can be
written as $H=\sum_k u_k u_k^*$, this is equivalent to positivity of
$\Phi$. By \eqref{geo:quadratic-chern}, we therefore seek a positive map
for which $F_n(\Phi)<0$.

For $f,g\in\mathbb C[x_1,\ldots,x_N]$, the apolar pairing is
\begin{equation}\label{alg:apolar-pairing}
 \langle f,g\rangle
   \coloneqq\left.f(\partial_x)g(x)\right|_{x=0}.
\end{equation}
With $x^\alpha=\prod_\nu x_\nu^{\alpha_\nu}$ and
$\alpha!=\prod_\nu\alpha_\nu!$ for multi-indices
$\alpha,\beta\in\mathbb Z_{\geq0}^N$, this pairing satisfies
\[
 \langle x^\alpha,x^\beta\rangle
  =\begin{cases}\alpha!,&\alpha=\beta,\\0,&\alpha\neq\beta.
    \end{cases}
\]
The pairing is symmetric, and therefore
\[
 F_n(\Phi)=\langle\det Y,\det\Phi(Y)\rangle
           =\langle\det\Phi(Y),\det Y\rangle.
\]
The next lemma realizes the apolar pairing as an expectation of
polynomials in independent standard complex Gaussian variables and
records the corresponding integration-by-parts identity.

\begin{lemma}\label{prob:gaussian-apolarity}
Let $N\geq1$ be an integer and $Z=(Z_1,\ldots,Z_N)$ have independent
standard complex Gaussian coordinates, each with density $\pi^{-1}e^{-|z|^2}$. For polynomials
$f,g,q\in\mathbb C[x_1,\ldots,x_N]$, the following identities hold.
\begin{enumerate}
\item[(i)] The Gaussian realization of the apolar pairing is
\begin{equation}\label{eq:gaussian-pairing}
 \langle f,g\rangle=\mathbb E[f(Z)g(\overline Z)].
\end{equation}
\item[(ii)] For every $1\leq\nu\leq N$, Gaussian integration by parts gives
\begin{equation}\label{eq:gaussian-ibp}
 \mathbb E[\overline{Z_\nu}f(Z)g(\overline Z)]
 =\mathbb E[(\partial_{x_\nu}f)(Z)g(\overline Z)].
\end{equation}
Consequently, multiplication transfers to differentiation as
\begin{equation}\label{alg:apolar-transfer}
 \langle f,qg\rangle=\langle q(\partial_x)f,g\rangle.
\end{equation}
\end{enumerate}
\end{lemma}

\begin{proof}
(i) Angular integration and the substitution $t=|z|^2$ give,
for nonnegative integers $a,b$,
\[
 \mathbb E[Z_\nu^a\overline{Z_\nu}^{\,b}]
 =\begin{cases}
 0,&a\neq b,\\
 \displaystyle\int_0^\infty t^ae^{-t}\,dt=a!,&a=b.
 \end{cases}
\]
Independence gives the corresponding identity for monomials in
$Z_1,\ldots,Z_N$, and bilinearity proves
\eqref{eq:gaussian-pairing} for arbitrary complex coefficients.

(ii) Write $z_\nu=u_\nu+iv_\nu$. Since
\[
 \overline z_\nu e^{-\|z\|^2}
 =-\frac12(\partial_{u_\nu}-i\partial_{v_\nu})e^{-\|z\|^2},
 \qquad
 (\partial_{u_\nu}-i\partial_{v_\nu})
   \bigl(f(z)g(\overline z)\bigr)
 =2(\partial_{x_\nu}f)(z)g(\overline z),
\]
ordinary integration by parts gives
\begin{align*}
 \mathbb E[\overline{Z_\nu}f(Z)g(\overline Z)]
 &=-\frac1{2\pi^N}\int_{\mathbb C^N}
   f(z)g(\overline z)
   (\partial_{u_\nu}-i\partial_{v_\nu})e^{-\|z\|^2}
   \,d\lambda(z)\\
 &=\frac1{2\pi^N}\int_{\mathbb C^N}
   (\partial_{u_\nu}-i\partial_{v_\nu})
   \bigl(f(z)g(\overline z)\bigr)e^{-\|z\|^2}
   \,d\lambda(z)\\
 &=\mathbb E[(\partial_{x_\nu}f)(Z)g(\overline Z)].
\end{align*}
Here $\lambda$ is Lebesgue measure on $\mathbb C^N$. Applying
\eqref{eq:gaussian-ibp} to the factors of each monomial in $q$ and using
linearity gives
\[
 \mathbb E[f(Z)q(\overline Z)g(\overline Z)]
 =\mathbb E[(q(\partial_x)f)(Z)g(\overline Z)].
\]
Together with \eqref{eq:gaussian-pairing}, this proves
\eqref{alg:apolar-transfer}.
\end{proof}

If $h\in\mathbb C[y_{ij}:1\leq i,j\leq n]$ is homogeneous of degree $n$,
then $\det(\partial_Y)h$ is constant. The symmetry of the
apolar pairing, the real coefficients of the determinant, and
Lemma~\ref{prob:gaussian-apolarity}(i) give
\[
 \det(\partial_Y)h
 =\langle h,\det Y\rangle
 =\mathbb E[h(G_n)\overline{\det G_n}].
\]
Here $G_n$ is an $n\times n$ random matrix with independent standard
complex Gaussian entries. Taking
$h=\det\Phi(Y)$ gives
\[
 F_n(\Phi)=\mathbb E\bigl[\overline{\det G_n}\det\Phi(G_n)\bigr],
\]
as in \eqref{eq:gaussian-F}.
In this expectation, a monomial has zero expectation whenever some
Gaussian variable and its conjugate occur with different powers.
This orthogonality identifies the terms that can contribute to $F_n$;
we seek a positive map for which their sum is negative.

The rank-nine map is built from the six-dimensional Breuer--Hall map of
Breuer \cite{Breuer2006} and Hall \cite[Theorem~4]{Hall2006}, which takes
two copies of the three-dimensional reduction map together with
skew-symmetric off-diagonal blocks and is positive. We enlarge it by a
three-dimensional diagonal block whose determinant occurs as a factor in
the determinant of the whole output. In the resulting rank-nine map,
$F_3(R)=0$ eliminates the terms independent of the skew-symmetric blocks,
so the negative Gaussian expectation comes from those blocks.

Write $\mathcal Y\in M_9(\mathbb C)$ as the following
$3\times3$ matrix of $3\times3$ blocks:
\[
 \mathcal Y=
 \begin{pmatrix}
  X&W&U_{13}\\
  V&Y&U_{23}\\
  U_{31}&U_{32}&Z
 \end{pmatrix},
 \qquad
 S\coloneqq\begin{pmatrix}X&W\\V&Y\end{pmatrix}.
\]
For
$J\coloneqq\begin{psmallmatrix}0&I_3\\-I_3&0\end{psmallmatrix}$,
consider
\begin{equation}\label{con:breuer-hall}
 B_6(S)\coloneqq\operatorname{tr}(S)I_6-S-JS^{\mathsf T}J^*.
\end{equation}
The map $B_6$ is the six-dimensional Breuer--Hall map.
We now consider the following family of maps on $M_9(\mathbb C)$:
\begin{equation}\label{con:map}
 \Phi_t(\mathcal Y)\coloneqq
 \begin{pmatrix}
  \begin{pmatrix}Z&0\\0&Z^{\mathsf T}\end{pmatrix}
       +tB_6(S)&0\\
  0&X+Y^{\mathsf T}
 \end{pmatrix},
 \qquad t\geq0.
\end{equation}
Transposition in
\eqref{con:breuer-hall} and \eqref{con:map} acts on matrix indices and
does not conjugate scalars, so both maps are complex-linear.

\begin{proposition}\label{con:positive-family}
For every real $t\geq0$, the map $\Phi_t$ in \eqref{con:map} is positive.
\end{proposition}

\begin{proof}
The matrix $J$ satisfies $J^{\mathsf T}=-J$ and $J^*J=I_6$.
Hence Hall's positivity theorem for the extended reduction map
\cite[Theorem~4]{Hall2006} applies to \eqref{con:breuer-hall} and gives
$B_6(S)\geq0$ whenever $S\geq0$.

If $\mathcal Y\geq0$, then $S,X,Y,Z\geq0$.
Together with
$B_6(S)\geq0$ and $t\geq0$, these inequalities prove
$\Phi_t(\mathcal Y)\geq0$.
\end{proof}

Proposition~\ref{con:positive-family} and
\eqref{geo:quadratic-curvature} show that $\Omega^{h_{\Phi_t}}|_0$ is
Griffiths semipositive. By \eqref{geo:quadratic-chern},
\[
 c_9(\mathbb C^9\times\mathbb C^9,h_{\Phi_t})|_0
 =F_9(\Phi_t)dV_9|_0.
\]
The desired separation of the two positivity conditions therefore
follows from the next exact evaluation.

\begin{theorem}\label{alg:main}
For every real $t\geq0$, the positive map $\Phi_t$ of
\eqref{con:map} satisfies
\begin{equation}\label{alg:main-value}
 F_9(\Phi_t)=-1536t^3.
\end{equation}
\end{theorem}

\begin{proof}
Set
\[
 A\coloneqq X+Y^{\mathsf T},\qquad
 K\coloneqq W^{\mathsf T}-W,\qquad L\coloneqq V^{\mathsf T}-V.
\]
A direct computation gives
\[
 JS^{\mathsf T}J^*
 =\begin{pmatrix}Y^{\mathsf T}&-W^{\mathsf T}\\
                 -V^{\mathsf T}&X^{\mathsf T}\end{pmatrix},
 \qquad
 \operatorname{tr}(S)=\operatorname{tr}(A),
\]
and therefore
\begin{equation}\label{con:block-form}
 B_6(S)=\begin{pmatrix}R(A)&K\\L&R(A)^{\mathsf T}\end{pmatrix},
 \qquad
 \Phi_t(\mathcal Y)
 =\begin{pmatrix}
   \begin{pmatrix}Z&0\\0&Z^{\mathsf T}\end{pmatrix}+tB_6(S)&0\\
   0&A
  \end{pmatrix},
\end{equation}
where $R(A)=\operatorname{tr}(A)I_3-A$ is the three-dimensional
reduction map on $M_3(\mathbb C)$. Consequently,
\[
 \det\Phi_t(\mathcal Y)
 =\det A\,
  \det\begin{pmatrix}
   Z+tR(A)&tK\\
   tL&Z^{\mathsf T}+tR(A)^{\mathsf T}
  \end{pmatrix}.
\]
For a column vector $a=(a_1,a_2,a_3)^{\mathsf T}$, the notation $[a]$
denotes the matrix
\begin{equation}\label{alg:skew-convention}
 [a]\coloneqq
 \begin{pmatrix}
 0&a_3&-a_2\\
 -a_3&0&a_1\\
 a_2&-a_1&0
 \end{pmatrix}.
\end{equation}
Write $W=(w_{ij})_{i,j=1}^3$ and $V=(v_{ij})_{i,j=1}^3$. Under the
convention \eqref{alg:skew-convention}, the skew-symmetric matrices
$K=W^{\mathsf T}-W$ and $L=V^{\mathsf T}-V$ are represented by the
column vectors
\begin{equation}\label{alg:reduced-variables}
 k\coloneqq
 \begin{pmatrix}w_{32}-w_{23}\\w_{13}-w_{31}\\w_{21}-w_{12}\end{pmatrix},
 \qquad
 l\coloneqq
 \begin{pmatrix}v_{32}-v_{23}\\v_{13}-v_{31}\\v_{21}-v_{12}\end{pmatrix}.
\end{equation}
Indeed, $[k]=K=W^{\mathsf T}-W$ and $[l]=L=V^{\mathsf T}-V$.
Set $M\coloneqq Z+tR(A)$. If $M$ is invertible, then
\eqref{alg:skew-convention} gives
\[
 M^{-\mathsf T}[l]M^{-1}=\frac{1}{\det M}[Ml],
 \qquad
 [Ml][k]=k(Ml)^{\mathsf T}-(k^{\mathsf T}Ml)I_3.
\]
The Schur complement and the matrix determinant lemma give
\begin{align}
 \det\Phi_t(\mathcal Y)
 &=(\det A)(\det M)^2
   \det\left(I_3-\frac{t^2}{\det M}[Ml][k]\right)\notag\\
 &=(\det A)\bigl(\det M+t^2k^{\mathsf T}Ml\bigr)^2\notag\\
 &=(\det A)\left[
    \det\bigl(Z+tR(A)\bigr)
       +t^2 k^{\mathsf T}\bigl(Z+tR(A)\bigr)l
         \right]^2.
 \label{alg:determinant-polynomial}
\end{align}
Both sides are polynomials in the entries of $M$, so the identity also
holds when $M$ is singular.

The polynomial in \eqref{alg:determinant-polynomial} depends only on
$\mathcal Y_{ij}$ with either $1\leq i,j\leq6$ or $7\leq i,j\leq9$.
Consequently, a term
$\prod_{i=1}^9\partial/\partial\mathcal Y_{i,\sigma(i)}$ in
$\det(\partial_{\mathcal Y})$ vanishes unless $\sigma$ preserves both
$\{1,\ldots,6\}$ and $\{7,8,9\}$. Thus only
$\sigma\in S_6\times S_3\subset S_9$ contribute, and on this polynomial
\begin{equation}\label{alg:block-differential-factorization}
 \det(\partial_{\mathcal Y})
 =\det(\partial_S)\det(\partial_Z).
\end{equation}

The polynomial $\det\Phi_t(\mathcal Y)$ is homogeneous of total degree nine
in the entries of $S$ and $Z$. Since $\det(\partial_Z)$ is a differential
operator of order three, $\det(\partial_Z)\det\Phi_t(\mathcal Y)$ is
homogeneous of total degree six in the entries of $S$ and $Z$. Every term
of positive degree in $Z$ has
degree at most five in the entries of $S$ and is therefore annihilated
by the differential operator $\det(\partial_S)$, which has order six.
Therefore,
\begin{equation}\label{alg:F-block-reduction}
 F_9(\Phi_t)
 =\det(\partial_S)\det(\partial_Z)\det\Phi_t(\mathcal Y)
 =\det(\partial_S)\!\left(
   \left.\det(\partial_Z)\det\Phi_t(\mathcal Y)\right|_{Z=0}
  \right).
\end{equation}

The ordinary and all-minors Cayley identities for $3\times3$ matrices
\cite[Theorem~2.1]{CaraccioloSokalSportiello2013} give
\begin{equation}\label{alg:cayley-identities}
 \det(\partial_M)\bigl((\det M)^2\bigr)=24\det M,
 \qquad
 \det(\partial_M)\bigl((\det M)M_{ij}\bigr)=8M_{ij}.
\end{equation}
The second formula follows by applying the product rule to the cofactor
case of the all-minors identity. Expanding the square now yields, for
$a,b\in\mathbb C^3$,
\[
 \det(\partial_M)\bigl(\det M+a^{\mathsf T}Mb\bigr)^2
 =24\det M+16a^{\mathsf T}Mb.
\]
Taking $M=Z+tR(A)$, $a=tk$, and $b=tl$ gives
\begin{align}
 \left.\det(\partial_Z)\det\Phi_t(\mathcal Y)\right|_{Z=0}
 &=\left.(\det A)\left[24\det\bigl(Z+tR(A)\bigr)
             +16t^2k^{\mathsf T}\bigl(Z+tR(A)\bigr)l\right]\right|_{Z=0}\notag\\
 &=t^3(\det A)\bigl[24\det R(A)+16k^{\mathsf T}R(A)l\bigr].
 \label{alg:after-Z}
\end{align}

To apply the remaining factor $\det(\partial_S)$ in
\eqref{alg:block-differential-factorization} to
\eqref{alg:after-Z}, write
\begin{equation}\label{alg:operators}
 \mathcal D\coloneqq\det(\partial_A),\qquad
 \mathcal L\coloneqq\sum_{i,j=1}^3
    \frac{\partial}{\partial A_{ij}}
    \frac{\partial}{\partial k_i}
    \frac{\partial}{\partial l_j}.
\end{equation}
On polynomials in the scalar variables $A_{ij},k_r,l_r$
$(1\leq i,j,r\leq3)$, regarded as polynomials in the entries of
$X,Y,W,V$ through $A=X+Y^{\mathsf T}$ and
\eqref{alg:reduced-variables}, one has
\begin{equation}\label{alg:chain-rule}
 \partial_X=\partial_A,\qquad
 \partial_Y=\partial_A^{\mathsf T},\qquad
 \partial_W=-[\partial_k],\qquad
 \partial_V=-[\partial_l].
\end{equation}
All these constant-coefficient differential operators commute. The
block determinant calculation in \eqref{alg:determinant-polynomial}
therefore gives
\begin{align}
 \det(\partial_S)
 &=\det\begin{pmatrix}
       \partial_A&-[\partial_k]\\
       -[\partial_l]&\partial_A^{\mathsf T}
      \end{pmatrix}\notag\\
 &=\left(\det(\partial_A)
       +(-\partial_k)^{\mathsf T}\partial_A(-\partial_l)\right)^2\notag\\
 &=\left(\det(\partial_A)+\sum_{i,j=1}^3
       \frac{\partial}{\partial A_{ij}}
       \frac{\partial}{\partial k_i}
       \frac{\partial}{\partial l_j}\right)^2\notag\\
 &=(\mathcal D+\mathcal L)^2
   =\mathcal D^2+2\mathcal D\mathcal L+\mathcal L^2.
 \label{alg:six-dimensional-differential}
\end{align}
Since $(\det A)\det R(A)$ is independent of $k,l$,
\[
 \mathcal L\bigl((\det A)\det R(A)\bigr)=0,
 \qquad
 (\mathcal D+\mathcal L)^2\bigl((\det A)\det R(A)\bigr)
 =\mathcal D^2\bigl((\det A)\det R(A)\bigr).
\]
For the second summand, writing
$k^{\mathsf T}R(A)l=\sum_{p,q=1}^3k_pR(A)_{pq}l_q$ gives
\[
 \mathcal L\bigl((\det A)k^{\mathsf T}R(A)l\bigr)
 =\sum_{i,j,p,q=1}^3
   \partial_{A_{ij}}\partial_{k_i}\partial_{l_j}
   \bigl((\det A)k_pR(A)_{pq}l_q\bigr)
 =\sum_{i,j=1}^3\partial_{A_{ij}}
   \bigl((\det A)R(A)_{ij}\bigr).
\]
Moreover, $(\det A)k^{\mathsf T}R(A)l$ has degree four in the entries of $A$,
whereas $\mathcal D^2$ has order six in those entries. Consequently,
\[
 (\mathcal D+\mathcal L)^2\bigl((\det A)k^{\mathsf T}R(A)l\bigr)
 =2\mathcal D\sum_{i,j=1}^3\partial_{A_{ij}}
   \bigl((\det A)R(A)_{ij}\bigr).
\]
By \eqref{alg:F-block-reduction}, \eqref{alg:after-Z}, and
\eqref{alg:six-dimensional-differential},
\begin{align}
 F_9(\Phi_t)
 &=t^3(\mathcal D+\mathcal L)^2\left(
    (\det A)\bigl[24\det R(A)+16k^{\mathsf T}R(A)l\bigr]
   \right)\notag\\
 &=t^3\biggl[
   24\mathcal D^2\bigl((\det A)\det R(A)\bigr)
    +32\mathcal D\sum_{p,q=1}^3
          \partial_{A_{pq}}\bigl((\det A)R(A)_{pq}\bigr)
                         \biggr].\label{alg:two-contributions}
\end{align}

We next compute in turn the two terms in
\eqref{alg:two-contributions}.
Writing $\tau\coloneqq\operatorname{tr}A$, the characteristic polynomial
of $A$ gives
\[
 \det R(A)=\tau e_2(A)-\det A,
 \qquad e_2(A)=\sum_{1\leq i<j\leq3}(A_{ii}A_{jj}-A_{ij}A_{ji}).
\]
A direct evaluation gives
$\mathcal D(\tau e_2(A))=\mathcal D(\det A)=6$, and hence
$\mathcal D\det R(A)=0$.
Equations~\eqref{alg:apolar-transfer} and
\eqref{alg:cayley-identities} now give
\begin{align}
 \mathcal D^2\bigl((\det A)\det R(A)\bigr)
 &=\langle(\det A)^2,(\det A)\det R(A)\rangle\notag\\
 &=\langle \mathcal D((\det A)^2),\det R(A)\rangle
  =24\mathcal D\det R(A)=0.
 \label{alg:first-contribution}
\end{align}

For the second term in \eqref{alg:two-contributions}, let
$G=(g_{pq})$ be a $3\times3$ random matrix with independent standard
complex Gaussian entries. By Lemma~\ref{prob:gaussian-apolarity}, one has
\begin{align}
 \mathcal D\sum_{p,q}\partial_{A_{pq}}\bigl((\det A)R(A)_{pq}\bigr)
 &=\sum_{p,q}\mathbb E\bigl[
    \overline{\det G}\,
    \partial_{A_{pq}}\bigl((\det A)R(A)_{pq}\bigr)(G)\bigr]\notag\\
 &=\mathbb E\bigl[|\det G|^2
       \sum_{p,q}\overline{g_{pq}}R(G)_{pq}\bigr]\notag\\
 &=\mathbb E\bigl[|\det G|^2
       (|\operatorname{tr}G|^2-\|G\|_{\mathrm F}^2)\bigr].
 \label{alg:gaussian-sign}
\end{align}
Here $\|G\|_{\mathrm F}^2=\sum_{p,q}|g_{pq}|^2$.
Lemma~\ref{prob:gaussian-apolarity}(i) and
\eqref{alg:cayley-identities} yield, for $1\leq p,q,r,s\leq3$,
\begin{align*}
 \mathbb E[|\det G|^2g_{pq}\overline{g_{rs}}]
 &=\langle(\det A)A_{rs},(\det A)A_{pq}\rangle\\
 &=\left.\partial_{A_{rs}}\mathcal D
   \bigl((\det A)A_{pq}\bigr)\right|_{A=0}
 =8\delta_{pr}\delta_{qs}.
\end{align*}
Thus the diagonal cross terms in the weighted trace moment vanish,
while every squared entry contributes to the weighted Frobenius moment:
\[
 \mathbb E[|\det G|^2|\operatorname{tr}G|^2]=24,
 \qquad
 \mathbb E[|\det G|^2\|G\|_{\mathrm F}^2]=72.
\]
Their difference determines the sign of the contribution from the two
skew-symmetric blocks in \eqref{alg:gaussian-sign}:
\begin{equation}\label{alg:gaussian-difference}
 \mathbb E\bigl[|\det G|^2
 (|\operatorname{tr}G|^2-\|G\|_{\mathrm F}^2)\bigr]
 =-48.
\end{equation}
Equations \eqref{alg:two-contributions}--\eqref{alg:gaussian-difference}
yield
\[
 F_9(\Phi_t)=-1536t^3,
\]
which proves \eqref{alg:main-value}.
\end{proof}

\section{Compact counterexamples from positive maps}\label{sec:geometric-counterexample}
\label{sec:strict}\label{sec:forms}\label{sec:metric}

We extend the rank-nine map to every $r\geq9$, make it strictly positive
by a trace perturbation, and realize the resulting map on projective space.

For $r\geq9$, $\eta\geq0$, and $Y=(y_{ij})_{i,j=1}^r\in M_r$, let $X$ be the leading
$9\times9$ principal submatrix of $Y$ and set
\begin{equation}\label{geo:extended-map}
 \Lambda_{r,\eta}(Y)\coloneqq
 \operatorname{diag}\bigl(\Phi_1(X),y_{10,10},\ldots,y_{rr}\bigr)
 +\eta\operatorname{tr}(Y)I_r,
\end{equation}
where the scalar diagonal entries are omitted when $r=9$.

\begin{proposition}\label{geo:strict-map}
For every integer $r\geq9$, there exists $\eta_0=\eta_0(r)>0$ such that
$\Lambda_{r,\eta}$ is strictly positive and
$F_r(\Lambda_{r,\eta})<0$ whenever $0<\eta<\eta_0$.
\end{proposition}

\begin{proof}
If $Y\geq0$, then $X\geq0$ and $y_{jj}\geq0$ for $j>9$.
Proposition~\ref{con:positive-family} therefore shows that
$\Lambda_{r,0}$ is positive. Moreover,
\[
 \det\Lambda_{r,0}(Y)=\det\Phi_1(X)\prod_{j=10}^ry_{jj}.
\]
By \eqref{eq:F-definition},
\[
 F_r(\Lambda_{r,0})=\det(\partial_X)\det\Phi_1(X)=F_9(\Phi_1)=-1536.
\]
For every nonzero $Y\geq0$ and every $\eta>0$,
\[
 \Lambda_{r,\eta}(Y)
 =\Lambda_{r,0}(Y)+\eta\operatorname{tr}(Y)I_r
 \geq\eta\operatorname{tr}(Y)I_r>0.
\]
Thus $\Lambda_{r,\eta}$ is strictly positive. Its coefficients are real,
so $F_r(\Lambda_{r,\eta})$ is a real polynomial in $\eta$. Since its value
at $\eta=0$ is $-1536$, it remains negative for all sufficiently small
$\eta>0$.
\end{proof}

For a strictly positive complex-linear map $\Lambda:M_n\to M_q$,
\eqref{geo:quadratic-metric} takes the form
\[
 h_\Lambda(z)\coloneqq I_q-2\pi\Lambda(zz^*).
\]
The following proposition realizes its curvature at a point by a global
Griffiths-positive metric. The realization can be confined to an arbitrary
neighborhood of that point.

\begin{proposition}\label{geo:projective-realization}
Let $n,q\geq1$ and $\Lambda:M_n\to M_q$ be a strictly positive
complex-linear map. The bundle $\mathcal O_{\mathbb P^n}(1)^{\oplus q}$
admits a smooth Griffiths-positive Hermitian metric $h$ such that, in the
standard affine coordinates and local frame at
$p\coloneqq[1:0:\cdots:0]$,
\[
 \Omega^h|_p
 =i\sum_{i,j=1}^n\Lambda(E_{ij})\,dz_i\wedge d\bar z_j\big|_p.
\]
Given any neighborhood $U$ of $p$, the metric can be chosen to agree with
the Fubini--Study direct sum metric outside $U$ and with $h_\Lambda$ near
$p$.
\end{proposition}

\begin{proof}
Fix a neighborhood $U$ of $p$. Strict positivity and compactness of the
product of unit spheres give a constant $a>0$ such that
\[
 \Lambda(uu^*)\geq a\|u\|^2I_q
 \qquad(u\in\C^n).
\]
On the standard affine chart, write $s\coloneqq\|z\|^2$ and let
\[
 h_0(z)\coloneqq\frac{I_q}{1+s},\qquad
 h_1(z)\coloneqq h_\Lambda(z),\qquad
 h_t\coloneqq(1-t)h_0+th_1\quad(0\leq t\leq1).
\]
For a constant tangent vector $u$, subscripts $u,\bar u$ denote
directional derivatives. At the origin,
\[
 h_t(0)=I_q,\qquad
 -(h_t)_{u\bar u}(0)
 =(1-t)\|u\|^2I_q+2\pi t\Lambda(uu^*)
 \geq\min\{1,2\pi a\}\|u\|^2I_q.
\]
Continuity and compactness in $t$ and in the unit tangent and fiber vectors
allow us to choose a coordinate ball $B_\rho$ with
$\overline B_\rho\subset U$ and a constant $c>0$ such that, throughout
$\overline B_\rho$,
\begin{equation}\label{geo:uniform-bounds}
 h_t\geq cI_q,
 \qquad
 -(h_t)_{u\bar u}\geq c\|u\|^2I_q
 \qquad(0\leq t\leq1).
\end{equation}

Set $H\coloneqq h_1-h_0$ and $\sigma\coloneqq z^*u$. Direct differentiation
gives
\[
 H(z)=\frac{s}{1+s}I_q-2\pi\Lambda(zz^*),
 \qquad
 H_u(z)=\frac{\sigma}{(1+s)^2}I_q-2\pi\Lambda(uz^*).
\]
Every linear map between finite-dimensional normed spaces is bounded.
Hence, for a constant $C>0$ independent of $z$ and $u$,
\[
 \|\Lambda(zz^*)\|_{\mathrm{op}}\leq Cs,
 \qquad
 \|\Lambda(uz^*)\|_{\mathrm{op}}\leq C\sqrt{s}\,\|u\|.
\]
Together with $|\sigma|\leq\sqrt{s}\,\|u\|$, these estimates give, for some
constant $C'>0$,
\[
 \begin{aligned}
 \|H(z)\|_{\mathrm{op}}
 &\leq\frac{s}{1+s}+2\pi\|\Lambda(zz^*)\|_{\mathrm{op}}
 \leq C's,\\
 \|H_u(z)\|_{\mathrm{op}}
 &\leq\frac{|\sigma|}{(1+s)^2}+2\pi\|\Lambda(uz^*)\|_{\mathrm{op}}
 \leq C'\sqrt{s}\,\|u\|
 \qquad(z\in\overline B_\rho).
 \end{aligned}
\]
Choose a smooth function $\theta:\mathbb R\to[0,1]$ which equals one on
$(-\infty,0]$ and zero on $[1,\infty)$. Given $L\geq1$, write
\[
 \chi_L(z)\coloneqq
 \begin{cases}
  \displaystyle\theta\left(1+\frac1L\log\frac{s}{\rho^2}\right),&z\neq0,\\[1.2ex]
  1,&z=0.
 \end{cases}
\]
It is smooth, equals one when $\|z\|\leq\rho e^{-L/2}$, and vanishes
when $\|z\|\geq\rho$. At $z\neq0$, one has
\[
 (\chi_L)_u=\frac{\theta'(\tau)\sigma}{Ls},
 \qquad
 (\chi_L)_{u\bar u}
 =\frac{\theta''(\tau)|\sigma|^2}{L^2s^2}
  +\frac{\theta'(\tau)}L
   \left(\frac{\|u\|^2}{s}-\frac{|\sigma|^2}{s^2}\right),
\]
where $\tau\coloneqq1+L^{-1}\log(s/\rho^2)$.
Since $|\sigma|^2\leq s\|u\|^2$, there is a constant $C''>0$,
depending only on $\theta$ and independent of $L,z,u$, such that
\[
 |(\chi_L)_u|\leq\frac{C''}{L\sqrt{s}}\|u\|,
 \qquad
 |(\chi_L)_{u\bar u}|\leq\frac{C''}{Ls}\|u\|^2.
\]

On $B_\rho$, take
\[
 h\coloneqq h_0+\chi_LH=(1-\chi_L)h_0+\chi_Lh_1.
\]
The first inequality in \eqref{geo:uniform-bounds} shows that $h$ is positive definite.
Differentiating $h$ gives
\[
 -h_{u\bar u}=-(1-\chi_L)(h_0)_{u\bar u}-\chi_L(h_1)_{u\bar u}-(\chi_L)_{u\bar u}H-(\chi_L)_uH_{\bar u}-(\chi_L)_{\bar u}H_u.
\]
Since $H_{\bar u}=H_u^*$, the sum of the last three terms is a Hermitian matrix
whose operator norm is at most $3C'C''L^{-1}\|u\|^2$. Choosing $L$ so that
$3C'C''/L<c/2$, the second
inequality in \eqref{geo:uniform-bounds} gives
\[
 -h_{u\bar u}\geq\frac c2\|u\|^2I_q.
\]
Since $\chi_L\equiv1$ on a neighborhood of $0$, one has $h=h_1$ there.
Hence the inequality also holds at $z=0$. Equation
\eqref{geo:curvature-coefficients} then gives
\[
 hR^h(u,\bar u)
 =-h_{u\bar u}+h_u^*h^{-1}h_u
 \geq\frac c2\|u\|^2I_q,
\]
so $h$ is Griffiths positive on $B_\rho$.

Because $\chi_L=0$ for $\|z\|\geq\rho$, $h$ extends smoothly to a global
Griffiths-positive metric that agrees with the Fubini--Study direct sum
metric outside $U$. Near $p$, $h$ coincides with $h_1=h_\Lambda$.
In particular, $h(p)=I_q$, $dh(p)=0$, and
$h_{i\bar j}(p)=-2\pi\Lambda(E_{ij})$, so
\eqref{geo:curvature-coefficients} gives
\[
 \Omega^h|_p
 =-\frac{i}{2\pi}\sum_{i,j=1}^n
   h_{i\bar j}(p)\,dz_i\wedge d\bar z_j\big|_p
 =i\sum_{i,j=1}^n
   \Lambda(E_{ij})\,dz_i\wedge d\bar z_j\big|_p.
\]
\end{proof}

\begin{proof}[Proof of Theorem~\ref{cor:ext-geometric}]
Fix $r\geq9$, choose $0<\eta<\eta_0(r)$ as in
Proposition~\ref{geo:strict-map}, and set
$\Lambda_r\coloneqq\Lambda_{r,\eta}$. Then $\Lambda_r$ is strictly
positive and $F_r(\Lambda_r)<0$.
Proposition~\ref{geo:projective-realization} gives a smooth
Griffiths-positive metric $h_r$ on
$E_r=\mathcal O_{\mathbb P^r}(1)^{\oplus r}$ whose normalized curvature
at $p=[1:0:\cdots:0]$ is
\[
 \Omega^{h_r}|_p
 =i\sum_{i,j=1}^r\Lambda_r(E_{ij})\,dz_i\wedge d\bar z_j\big|_p.
\]
Taking the determinant of this curvature matrix, as in
\eqref{geo:quadratic-chern}, gives
\[
 c_r(E_r,h_r)|_p=F_r(\Lambda_r)dV_r|_p<0.
\]
By continuity, the top Chern form is negative on a neighborhood of $p$.
\end{proof}

\section*{Acknowledgment}

The author is grateful to his advisor, Professor Song-Yan Xie, for his
guidance and continued support.

\section*{AI Use Disclosure}

The central mathematical insights and proof strategy are due to the
author. In particular, condition~(iv) in the first-jet characterization
of Griffiths positivity \cite[Theorem~1.1(iv)]{DuXieFirstJet2026} led
the author to investigate the problem through measures and expectations
and to direct AI tools in exploring this approach. AI tools also assisted
with the search for and computation of
concrete examples and with revision of the manuscript. The
author independently verified all mathematical statements and takes full
responsibility for the content of the paper.

\end{document}